\documentclass[11pt,a4paper]{amsart}
\usepackage{amsfonts}
\usepackage{epsfig}
\usepackage{graphicx}
\usepackage{amsmath}
\usepackage{amssymb}
\usepackage{colordvi}
\usepackage{times,colordvi,amsmath,epsfig,float,multicol,subfigure}

\usepackage{amsfonts, amssymb, amsmath, amscd, amsthm, txfonts, graphicx, color, enumerate}

\usepackage{url}

\usepackage{hyperref}

\usepackage{times,colordvi,float,multicol,subfigure}

\usepackage{ amscd, amsthm, txfonts, enumerate}

 \allowdisplaybreaks

\numberwithin{equation}{section}
\numberwithin{figure}{section}

\newtheorem{theorem}{Theorem}[section]

\newtheorem*{theorema}{Theorem A}
\newtheorem*{theoremb}{Theorem B}
\newtheorem*{theoremc}{Theorem C}
\newtheorem*{theoremd}{Theorem D}

\newtheorem{proposition}[theorem]{Proposition}
\newtheorem{lemma}[theorem]{Lemma}

\newtheorem{example}[theorem]{Example}

\newtheorem{remark}[theorem]{Remark}
\newtheorem{definition}[theorem]{Definition}

\usepackage{tikz}

\title[Some Remarks on Anosov Families]{Some Remarks on Anosov Families}

\author[J. Muentes]{Jeovanny Muentes}

 \address{Jeovanny de Jesus Muentes Acevedo,  Facultad de Ciencias B\'asicas,  Universidad Tecnol\'ogica de  Bolivar, Cartagena de Indias - Colombia}
\email{jmuentes@utb.edu.co}

\author[R. Ribeiro] {Raquel Ribeiro}
\address{Raquel Ribeiro Barroso, Instituto  de Matem\'atica e Estat\'istica, Universidade de S\~ao Paulo, Brazil}
\email{raquel.ribeiro.math@gmail.com}

\begin{document}

\begin{abstract}
We  study  Anosov families which are  sequences  of diffeomorphisms along compact Riemannian manifolds such that the tangent bundle split into expanding and contracting subspaces. In this paper we verify  that a certain class of  Anosov families: (i) admit canonical coordinates (ii)  are expansive, (iii) satisfy the shadowing property, and (iv) exhibit a Markov partition.

\end{abstract}




\maketitle

\section{Introduction}

An \emph{Anosov family} is a (biinfinite) sequence of diffeomorphisms along a sequence of compact Riemannian manifolds, with  an invariant sequence of splittings of the tangent bundle into expanding and contracting subspaces, and with a uniform upper bound for the contraction and lower bound for the expansion.

\medskip

 Anosov families (Definition \ref{anosovfamily}) were introduced by P. Arnoux and A. Fisher in \cite{Fisher}, motivated by generalizing the
notion of Anosov diffeomorphisms.
The authors  concentrated their studies on linear Anosov families on the two-torus. The  first goal was to get a natural notion of completion for the collection of   the set of all orientation-preserving linear Anosov diffeomorphisms on the two-torus (see \cite{Fisher}).   Young \cite{Young} proved  that families consisting of $C^{1+1}$ perturbations of an Anosov  diffeomorphism of class $C^2$ are Anosov families.
 In \cite{KawanL} and \cite{SHA0} the authors studied formulas for the entropy of  a  non-stationary subshifts of finite type.   Muentes  studied   the  stable and unstable manifold Theorem for Anosov families and the   structural stability of  Anosov families on compact Riemannian manifolds (see \cite{JeoLS, JeoOA, JeoSS}). Recently,  Chupeng Wu and Yunhua Zhou in \cite{Chupeng},  obtained   a  symbolic representation for Anosov families  given by a non-stationary subshift of finite type   (see also \cite{Fisher}, \cite{Liu}). In this work we will study some  properties related to hyperbolicity in the Anosov families.

\medskip

  {From the work done by Walter in \cite{Walters}, many attempts have been made to express the
concept of hyperbolicity in topological terms. Notions as  shadowing, expansiveness, canonical coordinates, Markov partitions, and others, were essential to achieve results related to   hyperbolicity}. In this work we will study exactly these  properties in  Anosov families. We are able to obtain such results only for Anosov families whose sequence are of $C^2$-diffeomorphisms with bounded  derivative. The first result is about the structure of canonical coordinates for Anosov families (Definition \ref{Canonical}).

\begin{theorema}Anosov families admit canonical coordinates.
\end{theorema}

We also investigate the \emph{expansiveness} property. In rough terms, the concept of expansiveness means  that if two points stay near for forward and
backward iterates, then they must be equal. In some sense, expansive systems can be considered chaotic
since they exhibit sensitivity to the initial conditions. The appropriate notion of expansiveness for sequences of diffeomorphisms  is given in Definition \ref{Expansive}. Our second result follows:

\begin{theoremb}Anosov families are expansive.
\end{theoremb}

\emph{Shadowing} was introduced by Anosov and it is central in hyperbolic dynamic. For
instance, it is fundamental in the proof of the $C^1$ structural
stability of  uniformly hyperbolic systems (see \cite{KH}, \cite{shub}). Roughly
speaking, it allows us to trace a set of points which looks like an
orbit, but with errors, by a true orbit. For practical
applications, we can suppose that a map $f$  is viewed as the orbit
realized in numerical calculation by computer, or in physical
experiments, thus it could have errors. Then shadowing property
allow us to ``correct'' this errors, finding a true evolution
which nicely approximates $f$.
Many authors have studied these properties and its relation with the hyperbolicity, for example, \cite{P2},   \cite{P}, \cite{S1}, \cite{P4}, and also, in many contexts, as in  \cite{Ribeiro03}, \cite{Ribeiro02} and \cite{Ribeiro04}. Thus, to decide which systems
possess the shadowing property is an important problem in
dynamics.
So, we  can ask:
\vspace{0.2cm}

Question: \emph{How would be shadowing for Anosov families?  Would Anosov families have any shadowing notion}?

\vspace{0.2cm}

The precise definition of shadowing  for  Anosov families is  in Section \ref{SectionShadowing}. We conclude our third result.

\begin{theoremc} Anosov families have shadowing property.
\end{theoremc}

Shadowing has many applications and one of them is to get a Markov partition \cite{Z01}. In  Section 3.4 of \cite{Fisher05}, Arnoux and Fisher, gave a
 symbolic representation for an Anosov family that admits  a Markov partition
sequence. In this case the symbolic representation is given by a \textit{non-stationary subshifts of finite type}, which  was first investigated in \cite{Fisher05} with the motivation to study  Anosov families via coding and to deduce properties of adic transformations. In this paper we study   Markov partitions (Definition \ref{Markov}) for Anosov families. We consider  Anosov families whose  the sequence of manifold is constant, that is, the manifolds are equal, and for these families  we prove our fourth and last result:


\begin{theoremd}Anosov family has Markov partition.
\end{theoremd}

This article is organized as follows: In Section 2 we will  define precisely an  Anosov family and the objects that we will study in this work. We will make  important considerations, notations and comments  which are relevant in this context. In Section 3,  we  will mention the Stable and Unstable Manifold Theorem for Anosov families, proved by the first author in \cite{JeoLS}.
Sections 2 and 3 will ease the understanding of the behavior of the Anosov families. We will prove Theorem A in Section 4,  which is essential to obtain shadowing  for Anosov families. Theorems B and C will be proved in Section 5. In Section 6, we prove Theorem D, that is, in certain contexts there is a Markov partition for  Anosov families. We reserve the last section, Section 7, to propose future issues, and further generalizations about Anosov families.

 \section{Anosov Families and Definitions} \label{SectionAnosovFamily}


In this section we will introduce  Anosov family and we will mention the main elements that will be used throughout this work. In addition, we will give some examples and observations  of this class of systems. Firstly we will define the objects that are part of the context in which Anosov families are introduced.

\medskip

For $i\in \mathbb{Z}$, consider  a sequence of Riemannian manifolds   $M_{i}$   with a fixed Riemannian metric $\langle \cdot, \cdot\rangle_{i}$  and injectivity radius $\varrho_{i}$. We will suppose that $\varrho=\underset{i\in\mathbb{Z}}{\inf}\varrho_{i}>0$   (see \cite{JeoSS}, Remark 2.7). Take the  \textit{disjoint union}  $$\mathcal{M}=\coprod_{i\in \mathbb{Z}}{M_{i}}=\bigcup_{i\in \mathbb{Z}}{M_{i}\times{i}}.$$
   $\mathcal{M}$ will be endowed with the Riemannian metric $\langle \cdot, \cdot\rangle$    induced by  $\langle \cdot, \cdot\rangle_{i} $, setting
\(\langle \cdot, \cdot\rangle|_{M_{i}}=\langle \cdot, \cdot\rangle_{i} \)  for \(i\in \mathbb{Z}.\)  We denote by $\Vert \cdot\Vert_{i}$ the induced norm by  $\langle\cdot,\cdot\rangle_{i}$ on $TM_{i}$ and we will take   $\Vert \cdot \Vert$ defined on  $\mathcal{M}$  as    $\Vert \cdot\Vert|_{M_{i}}=\Vert \cdot\Vert_{i} $ for $i\in \mathbb{Z}$.

\begin{definition}\label{leidecomposicao} A  \emph{non-stationary dynamical system}  $(\mathcal{M},\langle\cdot,\cdot\rangle, \mathcal{F})$  is a map $\mathcal{F}:\mathcal{M}\rightarrow \mathcal{M}$, such that, for each $i\in\mathbb{Z}$, $\mathcal{F}|_{M_{i}}=f_{i}:M_{i}\rightarrow M_{i+1}$ is a  diffeomorphism. Sometimes we use the notation   $\mathcal{F}=(f_{i})_{i\in\mathbb{Z}}$. The composition law is defined   to be
\begin{equation*}
\mathcal{F}_{i} ^{\, n}:=
\begin{cases}
f_{i+n-1}\circ \cdots\circ f_{i}:M_{i}\rightarrow M_{i+n}  & \mbox{if }n>0 \\
  f_{i-n}^{-1}\circ \cdots\circ f_{i-1}^{-1}:M_{i}\rightarrow M_{i-n} & \mbox{if } n<0  \\
	 I_{i}:M_{i}\rightarrow M_{i}  & \mbox{if } n=0,\\
 \end{cases}
\end{equation*}
where $I_{i}$ is the identity on $M_{i}$.
\end{definition}

Non-stationary dynamical systems are classified via  \textit{topological equiconjugacy}:

\begin{definition}\label{definconjugacy} A \emph{topological equiconjugacy} between  $ \mathcal{F}=(f_{i})_{i\in\mathbb{Z}}$  and  $ \mathcal{G}=(g_{i})_{i\in\mathbb{Z}}$
is a map $\mathcal{H}:\mathcal{M}\rightarrow \mathcal{M}$, such that, for  each $i\in \mathbb{Z} ,$ $\mathcal{H}|_{M_{i}}=h_{i}:M_{i}\rightarrow M_{i}$ is a homeomorphism, $(h_{i})_{i\in\mathbb{Z}}$  and $(h_{i}^{-1})_{i\in\mathbb{Z}}$ are equicontinuous and     \(h_{i+1}\circ f_{i}=g_{i}\circ h_{i} .\)
In that case, we will say the families are \textit{equiconjugate}.
\end{definition}

Now, we have all the elements to rigorously define an Anosov family.

 \begin{definition}\label{anosovfamily}    An  \emph{Anosov family} on $\mathcal{M}$ is  a non-stationary dynamical system     $(\mathcal{M},\langle\cdot,\cdot\rangle, \mathcal{F})$ such that:
\begin{enumerate}[i.]
\item the tangent bundle $T\mathcal{M}$ has a continuous splitting   $E^{s}\oplus E^{u}$ which is  $D\mathcal{F}$-\textit{invariant}, i. e., for each $p\in \mathcal{M}$,
 $T_{p}\mathcal{M}=E^{s}_{p}\oplus E^{u}_{p}$ with $D_{p}\mathcal{F}(E^{s}_{p})= E^{s}_{\mathcal{F}(p)}$ and $D_{p}\mathcal{F} (E^{u}_{p})= E^{u}_{\mathcal{F}(p)}$, where $T_{p}\mathcal{M}$ is the    tangent space at $p;$
\item there exist constants $\lambda \in (0,1)$ and $c>0$ such that for each  $i\in \mathbb{Z}$, $n\geq 1$,    and $p\in M_{i}$,
we have: \[\Vert D_{p} (\mathcal{F}_{i}^{n})(v)\Vert \leq c\lambda^{n}\Vert v\Vert \text{ if   }v\in E_{p}^{s}\quad\text{and}\quad \Vert D_{p} (\mathcal{F}_{i}^{-n}) (v)\Vert \leq c\lambda^{n}\Vert v\Vert  \text{ if }v\in E_{p}^{u}.\]
\end{enumerate}
The  subspaces $E^{s}_{p}$ and $E^{u}_{p}$ are called     \textit{stable} and \textit{unstable} subspaces, respectively.
\end{definition}

If we can take
$c=1$ we say the family is \textit{strictly Anosov}.

\medskip

The next example, which is due to Arnoux and Fisher \cite{Fisher05}, Example 3, proves that Anosov families are not  necessarily sequences of Anosov diffeomorphisms. A random version of the example can be found in \cite{GUNDLACH}, Example 2.7. More examples can be found in  \cite{Fisher05, JeoLS,JeoSS}.

\begin{example}\label{multipl}    For any sequence of positive integers  $(n_{i})_{i\in \mathbb{Z}}$ set
 \begin{equation*} A_{i} =
\left(
\begin{array}{ccc}
1 & 0 \\
n_{i} & 1
 \end{array}
\right) \text{ for $i$ even}\quad\text{ and } \quad A_{i} =
\left(
\begin{array}{ccc}
1 & n_{i} \\
0 & 1
 \end{array}
\right) \text{ for $i$ odd},
\end{equation*}
acting on the 2-torus $M_{i}=\mathbb{T}^{2}$. The family $(A_i)_{i\in\mathbb{Z}}$ is an Anosov family.  \end{example}

 \begin{definition}\label{propang2} An Anosov family satisfies the \emph{property of the angles} (or \emph{s.p.a.}) if the angle between the stable and unstable subspaces are bounded away from zero (see \cite{JeoLS,JeoOA,JeoSS}).
 \end{definition}

\begin{remark}Fix an Anosov diffeomorphism $\phi$ on a Riemannian manifold $M$. For each $i\in\mathbb{Z}$, we can endow  $M_{i}=M$ with  a suitable Riemannian metric such that if we  consider  $f_{i}=\phi$ for any $i\in\mathbb{Z}$, then $(f_{i})_{i\in\mathbb{Z}}$ is an Anosov family such that the angle between the unstable and stable subspaces at some points of $M$ converges to zero as $i\rightarrow \infty$ (see \cite{JeoLS}, Example 2.4). That is, there exist Anosov families which do not satisfy the property of angles.
\end{remark}

Now we define some important sets which we will use throughout this work. Fix $m\geq 1$. The set   \[\mathcal{D}^{m}(\mathcal{M})=\{\mathcal{F}=(f_{i})_{i\in\mathbb{Z}}: f_{i}:M_{i}\rightarrow M_{i+1} \text{ is a  }C^{m}\text{-diffeomorphism}\}\]
can be endowed with the \textit{strong topology} and the \textit{uniform topology} (see \cite{JeoOA,JeoSS}). The subset of $\mathcal{D}^{m}(\mathcal{M})$ consisting   of Anosov families will be denoted by $\mathcal{A}^{m}(\mathcal{M})$.

  \medskip

Consider the set \[\mathcal{A}^{2}_{b}(\mathcal{M})=\{\mathcal{F}=(f_{i})_{i\in\mathbb{Z}}\in \mathcal{D}^{2}(\mathcal{M}): \mathcal{F}  \text{ is Anosov, s.p.a.  and }\sup_{i\in\mathbb{Z}}\Vert Df_{i}\Vert_{C^{2}}<\infty\},\]
where \( \Vert \phi\Vert_{C^{2}}= \max\left\{ \Vert D \phi\Vert,  \Vert D\phi^{-1}\Vert,  \Vert D^{2}\phi\Vert, \Vert D^{2}\phi^{-1}\Vert \right\} \)   for a $ C^{2}$-diffeomorphism \(\phi.\)




\section{Stable and Unstable Manifolds for Anosov Families}

In \cite{JeoLS} the author proved the local unstable and stable manifold theorem for Anosov family. This theorem is essential to prove Theorem A. In this section we will state the results that will be used in the next section. Firstly we note:

\begin{remark}In this section, we will consider $\mathcal{F}=(f_{i})_{i\in\mathbb{Z}}\in\mathcal{A}^{2}_{b}(\mathcal{M})$.\end{remark}

Now we define some sets which will be used throughout the work. Given $\varepsilon>0$ and $p\in\mathcal{M}$, set:
\begin{enumerate}[\upshape (i)]
\item $B(p, \varepsilon)\subseteq \mathcal{M} $   be the ball with radius $\varepsilon $ and center $ p$;
\item $B(\tilde{0}_{p},\varepsilon)\subseteq T_{p} \mathcal{M}$  denote the ball with radius $\varepsilon $ and center $\tilde{0}_{p}$, the zero vector in $T_{p} \mathcal{M}$;
\item $B^{s}(\tilde{0}_{p},\varepsilon)\subseteq E_{p}^{s} $  denote the ball with radius $\varepsilon $ and center $\tilde{0}_{p}$;
\item $B^{u}(\tilde{0}_{p},\varepsilon)\subseteq E_{p}^{u}$  denote the ball with radius $\varepsilon $ and center $\tilde{0}_{p}$.\end{enumerate}

\medskip

    Given two points $p,q\in \mathcal{M}$, set
\begin{align*}
\Theta_{p,q}&= \underset{n\rightarrow \infty}\limsup  \frac{1}{n}\log d(\mathcal{F}_{i}^{n}(q),\mathcal{F}_{i}^{n}(p))\quad  \text{and}\quad
 \Delta_{p,q}   =\underset{n\rightarrow \infty}\limsup\frac{1}{n}\log d(\mathcal{F}_{i}^{-n}(q),\mathcal{F}_{i}^{-n}(p)).
\end{align*}

\begin{definition}\label{conjuntosestaviesfam}  Let  $\varepsilon>0$. Fix  $p\in \mathcal{M}$.
\begin{enumerate}[\upshape (i)]
\item  \(\mathcal{W}^{s}(p,\varepsilon)=\{q\in B(p,\varepsilon):  \Theta_{p,q}<0\text{  and }\mathcal{F}_{i}^{\, n}(q)\in B(\mathcal{F}_{i}^{\, n}(p),\varepsilon)\text{ for }n\geq1\}:=\) the \textit{local stable set at} $p$;
\item  \(\mathcal{W}^{u}(p,\varepsilon)=\{q\in B(p,\varepsilon):
  \Delta_{p,q}<0\text{  and }\mathcal{F}_{i}^{-n}(q)\in B(\mathcal{F}_{i}^{-n}(p),\varepsilon)\text{ for }n\geq1\}:=\) the \textit{local unstable set at} $p$.
\end{enumerate}
\end{definition}

Since $\mathcal{F}$ satisfies the property of angles, we can suppose that  $\mathcal{F}$ is strictly Anosov and furthermore that  $E_{p}^{s}$ and $E_{p}^{u}$ are orthogonal for any $p\in\mathcal{M}$ (see \cite{JeoLS}). This is the Lemma of Mather for Anosov families.  In Theorems 5.2 and 5.3 of \cite{JeoLS} and  Theorems 3.7, 3.8, 4.5 and 4.6 of \cite{JeoSS} we proved that, for any   $ \alpha \in (0,(\lambda ^{-1}-1)/2 )$, there exist a small $\epsilon>0$  and     $\zeta\in (0,1)$    such that follow the next two results:

\begin{theorem}\label{variedadeinstave} For   each $p\in \mathcal{M}$,     $\mathcal{W}^{u}(p,\epsilon)$  is a differentiable submanifold of $\mathcal{M}$ and there exists $K^{u}>0$ such that:
 \begin{enumerate}[\upshape (i)]
 \item \(\text{exp}_{p}^{-1} (\mathcal{W} ^{u}(p,\epsilon))=\{(\phi_{p}^{u}(x),x): x\in B^{u}(\tilde{0}_{p},\epsilon) \},\) where $ \phi_{p}^{u}: B^{u}(\tilde{0}_{p},\epsilon)\rightarrow B^{s}(\tilde{0}_{p},\epsilon)$ is   an  $\alpha$-Lipschitz map and $ \phi_{p}^{u}(\tilde{0}_{p})=\tilde{0}_{p}.$
\item   $T_{p}\mathcal{W} ^{u}(p,\epsilon)=E_{p} ^{u}$,
\item $\mathcal{F}^{-1}(\mathcal{W}^{u}(p,\epsilon))\subseteq \mathcal{W}^{u}( \mathcal{F}  ^{\, -1}(p),\epsilon)$,
\item if $q\in \mathcal{W} ^{u}(p,\epsilon)$ and $n\geq1$ we have
\( d(\mathcal{F}^{-n}(q),\mathcal{F}^{-n}(p))\leq  K^{u}\zeta^{n}d(q,p).\)
\item Let $(p_{m})_{m\in \mathbb{N}}$   be a  sequence in $ M_{i}$ converging to $p\in M_{i}$ as $m\rightarrow \infty$. If  $q_{m}\in   \mathcal{W} ^{u}(p_{m},\epsilon)$ converges to $q\in B (p,\epsilon)$ as $m\rightarrow \infty$,  then $q\in \mathcal{W}^{u}(p,\epsilon)$.
\end{enumerate}
\end{theorem}

 \begin{theorem}\label{variedadeestavel}   For   each $p\in \mathcal{M}$,     $\mathcal{W}^{s}(p,\epsilon)$  is a differentiable submanifold of $\mathcal{M}$ and there exists $K^{s}>0$ such that:
\begin{enumerate}[\upshape (i)]
\item \(\text{exp}_{p}^{-1} (\mathcal{W} ^{s}(p,\epsilon))=\{(x,\phi_{p}^{s}(x)): x\in B^{s}(\tilde{0}_{p},\epsilon) \},\)
where $\phi_{p}^{s}: B^{s}(\tilde{0}_{p},\epsilon)\rightarrow B^{u}(\tilde{0}_{p},\epsilon)$ is an $\alpha$-Lipschitz map and $\phi_{p}^{s}(\tilde{0}_{p})=\tilde{0}_{p}.$
\item   $T_{p}\mathcal{W}^{s}(p,\epsilon)=E_{p} ^{s}$,
\item $\mathcal{F}  (\mathcal{W} ^{s}(p,\epsilon))\subseteq \mathcal{W}^{s}(\mathcal{F}(p),\epsilon)$,
\item if $q\in \mathcal{W} ^{s}(p,\epsilon)$ and $n\geq1$ we have
\(d(\mathcal{F}^{\, n}(q),\mathcal{F}^{\, n}(p))\leq  K^{s}\zeta^{n}d(q,p).
\)
\item Let $(p_{m})_{m\in \mathbb{N}}$   be a  sequence in $ M_{i}$ converging to $p\in M_{i}$ as $m\rightarrow \infty$. If  $q_{m}\in   \mathcal{W} ^{s}(p_{m},\epsilon)$ converges to $q\in B (p,\epsilon)$ as $m\rightarrow \infty$,  as $q\in \mathcal{W}^{s}(p,\epsilon)$.
\end{enumerate}
\end{theorem}

 Other property of the invariant manifolds for Anosov families follows in the proposition below.

\begin{proposition}\label{dasddd}
 Let $\beta \in (0,\epsilon/2)$. If $d(\mathcal{F}^{n}(p),\mathcal{F}^{n}(q))<\beta $ for all $n \in \mathbb{N}$, then $ q\in \mathcal{W}^{s}(p,\epsilon)$. On the other hand, if $d(\mathcal{F}^{-n}(p),\mathcal{F}^{-n}(q))<\beta $ for all $n \in \mathbb{N}$, then $ q\in \mathcal{W}^{u}(p,\epsilon)$.
\end{proposition}
\begin{proof}
By abuse of notation, we identify $\mathcal{W}^{s}( \mathcal{F}_{0} ^{n} (p),\epsilon )\times \mathcal{W}^{u}(\mathcal{F}_{0} ^{n}(p),\epsilon)$ with an open neighborhood  of $\tilde{0}\in T_{\mathcal{F}_{0} ^{n}(p)}\mathcal{M}$ via exponential charts.   Suppose that $q\notin  \mathcal{W}^{s}(p,\epsilon)$. Therefore, since $\mathcal{F}_{0} ^{n}(q)\in B  (\mathcal{F}_{0} ^{n}(p),\beta)$, we have     \[\mathcal{F}_{0} ^{n}  (\text{exp}_{p} ^{-1}(q) )=(x_{n},y_{n})\in \mathcal{W}^{s}( \mathcal{F}_{0} ^{n} (p),\epsilon )\times \mathcal{W}^{u}(\mathcal{F}_{0} ^{n}(p),\epsilon),  \]
 for all $n\geq0,$ with $x_{n}\in \mathcal{W}^{s}( \mathcal{F}_{0} ^{n}  (p)  ,\epsilon) $ and $y_{n} \in \mathcal{W}^{u}(  \mathcal{F}_{0} ^{n}(p) ,\epsilon )\setminus\{0\}$. We can obtain from Theorem \ref{variedadeinstave}, item (iv), and   Theorem \ref{variedadeestavel}, item (iv),  that   \begin{equation*}\Vert (x_{n}, y_{n})\Vert    \geq  \Vert x_{n}\Vert -  \Vert y_{n}\Vert  \geq \frac{1}{K^{u}\zeta^{n}}   \Vert y_{0}\Vert  - {K^{s}}{\zeta^{n}}   \Vert x_{0}\Vert.
\end{equation*}
We have  $ {K^{s}}{\zeta^{n}}   \Vert x_{0}\Vert \rightarrow 0$ as $n\rightarrow +\infty.$  Since $y_{0}\neq 0$,  for some $n\in\mathbb{N}$ we have
\(d(\mathcal{F}_{0} ^{n}(q),\mathcal{F}_{0} ^{n}(p))=\Vert (x_{n},  y_{n})\Vert   >\beta,\)
which contradicts the assumption.
Analogously we can prove the second part of the proposition.
\end{proof}

\section{Canonical Coordinates for Anosov Families} \label{SectionCanonical}

Canonical coordinates were introduced by
Bowen who used them to study Axiom A diffeomorphisms \cite{Z01,Z02,Z03}. He exploited
the fact that an Axiom A diffeomorphism restricted to a basic set has hyperbolic
canonical coordinates with respect to some metric. Other results related to canonical coordinates have been developed, such as Fathi's, which  says that an expansive homeomorphism on a compact metric space with canonical coordinates admits a metric compatible with the original topology to which the canonical coordinates are hyperbolic \cite{Fathi}. In our case, we will prove that Anosov families in $\mathcal{A}^{2}_{b}(\mathcal{M})$   admit canonical coordinates.

 \begin{definition} \label{Canonical} An Anosov family  $\mathcal{F}$ has  \emph{canonical coordinates} if  given a small $\epsilon>0$ there exists a $\delta>0$ such that, if $p,q\in \mathcal{M}$ with $d(p,q)<\delta$, then $$\mathcal{W}^{s}(p,\epsilon)\cap \mathcal{W}^{u}(q,\epsilon) \neq \emptyset.$$
 \end{definition}

In the single case,  the existence of canonical coordinates for Anosov families is a direct consequence of the continuity of the stable and unstable manifolds and the compactness of the manifold. However, for Anosov families, this fact is not immediate, because $ \mathcal{M}$  is not compact.

\begin{theorema}\label{coordenadas} Let $\mathcal{F}\in\mathcal{A}^{2}_{b}(\mathcal{M})$. Given a small $\epsilon>0$ there exists $\delta>0$ such that if $p,q \in \mathcal{M}$, and $d(p,q)<\delta$ then  $$\mathcal{W}^{s}(p,\epsilon)\cap \mathcal{W}^{u}(q,\epsilon)  $$ is a single point in $\mathcal{M}. $\end{theorema}

\begin{proof}
  Take $\varepsilon\in (0, \varrho)$, where $\varrho$ is an injectivity radius of $\mathcal{M}$. Hence,  for any $p\in\mathcal{M}$  the exponential map \[\text{exp}_{p}:B(\tilde{0}_{p},\varepsilon) \rightarrow B(p,\varepsilon) \] is a diffeomorphism and \( \Vert v\Vert  =d(\text{exp}_{p}(v), p)),\)  for all  \(v\in B(\tilde{0}_{p},\varepsilon).\)
  By Theorems \ref{variedadeinstave} and \ref{variedadeestavel}  we have that for any   $ \alpha \in (0,(\lambda ^{-1}-1)/2 )$ there exists an $\epsilon\in (0, \varrho/4)$ such that
\[\text{exp}_{p}^{-1} (\mathcal{W} ^{s}(p,\epsilon))=\{(x,\phi_{p}^{s}(x)): x\in B^{s}(\tilde{0}_{p},\epsilon) \} \] and\[\text{exp}_{p}^{-1} (\mathcal{W} ^{u}(p,\epsilon))=\{(\phi_{p}^{u}(x),x): x\in B^{u}(\tilde{0}_{p},\epsilon) \},\]
where $\phi_{p}^{s}: B^{s}(\tilde{0}_{p},\epsilon)\rightarrow B^{u}(\tilde{0}_{p},\epsilon)$ and $\phi_{p}^{u}: B^{u}(\tilde{0}_{p},\epsilon)\rightarrow B^{s}(\tilde{0}_{p},\epsilon)$ are $\alpha$-Lipschitz maps and $\phi_{p}^{s}(\tilde{0}_{p})=\phi_{p}^{u}(\tilde{0}_{p})=\tilde{0}_{p}.$ Set \begin{align*}K_{\alpha,p}^{s} =\{(v,w)\in E_{p}^{s}\oplus E_{p}^{u}: \Vert w\Vert \leq \alpha \Vert v \Vert \}
\quad \text{ and }\quad K_{\alpha ,p}^{u} =\{(v,w)\in E_{p}^{s}\oplus E_{p}^{u}: \Vert v\Vert\leq \alpha \Vert w\Vert \}  .
\end{align*}
Note that \[\text{exp}_{p}^{-1} (\mathcal{W} ^{s}(p,\epsilon))\subseteq K_{\alpha,p}^{s} \quad \text{and}\quad \text{exp}_{p}^{-1} (\mathcal{W} ^{u}(p,\epsilon))\subseteq K_{\alpha,p}^{u}.\]
Take $p,q\in\mathcal{M}$ with $d(p,q)<\epsilon/4$. Thus $\text{exp}_{p}^{-1}(q)\in B^{s}(\tilde{0}_{p},\epsilon)\times B^{u}(\tilde{0}_{p},\epsilon)$. Set $z=\text{exp}_{p}^{-1}(q)$, \[ \tilde{E}^{s}_{q} :=z+D (\text{exp}_{p}^{-1})_{q} (E_{q}^{s})\subseteq T_{p}\mathcal{M}\quad \text{and}\quad \tilde{E}^{u}_{q}:=z+D(\text{exp}_{p}^{-1}) _{q}(E_{q}^{u})\subseteq T_{p}\mathcal{M}.\]
Thus $\tilde{E}^{s}_{q}$  is parallel to   $E^{s}_{p}$    and  $\tilde{E}^{u}_{q}$ is parallel to   $E^{u}_{p}$.  Hence  $\tilde{E}^{s}_{q}$  is perpendicular to $E^{u}_{p}$ and $\tilde{E}^{u}_{q}$ is perpendicular to $E^{s}_{p}$    (remember that $E^{s}_{p}$    and $E^{u}_{p}$    are orthogonal).   Consequently, we can choose a  $\delta\in(0,\epsilon/4)$ small enough such that,  if $d(p,q)<\delta$, then  any $(u_{q},v_{q})\in \text{exp}_{p}^{-1}(q)+D (\text{exp}_{p}^{-1})_{q}( K_{\alpha,q}^{u})$ with $u_{q}\in B^{u}(\tilde{0}_{p},\epsilon)$ belongs to $B^{s}(\tilde{0}_{p},\epsilon)\times B^{u}(\tilde{0}_{p},\epsilon)$ (see Figure \ref{figure1}). 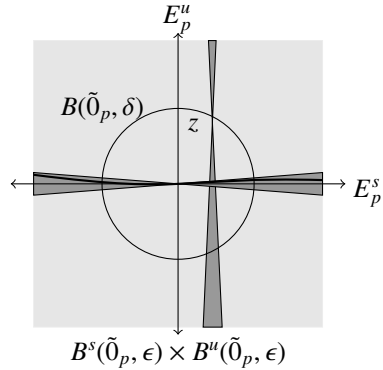
\begin{figure}[ht]
 \begin{center}
\begin{tikzpicture}
\draw[black!7,fill=black!10, ultra thin] (-7.9,-0.4)rectangle (-4.1,3.4);

\draw[black, fill=black!40, thin] (-5.6,3.4) -- (-5.5,3.4) -- (-5.67,-0.4) --   (-5.42,-0.4) -- (-5.6,3.4) -- cycle;

 \draw[black, fill=black!40, thin]  (-7.9,1.65) -- (-7.9,1.35) -- (-4.1,1.65) --  (-4.1,1.35) --  (-7.9,1.65) --  cycle;

 \draw (-6 , 1.5) circle (1cm);

\draw[<->] (-6,-0.5) -- (-6,3.5);
\draw[<->] (-8.2,1.5) -- (-3.8,1.5);
\draw (-6.2,4) node[below] {\quad{\small $E_{p}^{u}$}};
\draw (-3.7,1.7) node[below] {\quad{\small $E_{p}^{s}$}};
\draw (-6,-0.4) node[below] {\small $B^{s}(\tilde{0}_{p},\epsilon)\times B^{u}(\tilde{0}_{p},\epsilon)$};
\draw (-5.8,2.5) node[below] {\small $z$};

 \draw (-7,2.85) node[below] {\small $B(\tilde{0}_{p},\delta)$};
\draw[thick] (-7.9,1.62) .. controls (-5.8,1.35) and (-6,1.6) .. (-4.1,1.55);
\end{tikzpicture}
\end{center}
\caption{$z=\text{exp}_{p}^{-1}(q) $; the vertical cone is $\text{exp}_{p}^{-1}(q)+D (\text{exp}_{p}^{-1})_{q}( K_{\alpha,q}^{u})$;  the horizontal cone is $K_{\alpha,p}^{u}$; the curve inside $K_{\alpha,p}^{u}$ is $\mathcal{W} ^{s}(p,\epsilon)$.} \label{figure1}
\end{figure}

Therefore   $$[\text{exp}_{p}^{-1}(q)+D (\text{exp}_{p}^{-1})_{q}( K_{\alpha,q}^{u})]\cap  K_{\alpha,p}^{u}$$ is not empty and lives inside $B^{s}(\tilde{0}_{p},\epsilon)\times B^{u}(\tilde{0}_{p},\epsilon)$.
Since $ \text{exp}_{q}^{-1} (\mathcal{W} ^{u}(q,\epsilon))\subseteq K_{\alpha,q}^{u}$,  we have that $\mathcal{W} ^{s}(p,\epsilon)\cap \mathcal{W} ^{u}(q,\epsilon)$ is a single point in $\mathcal{M}.$  \end{proof}





 \section{Expansiveness and Shadowing}\label{SectionShadowing}

This section is divided into two subsections. In the first subsection we will verify  that the Anosov families in $\mathcal{A}^{2}_{b}(\mathcal{M})$ are expansive. In the second subsection we will use the main results from the previous sections, that is,  Anosov families in $\mathcal{A}^{2}_{b}(\mathcal{M})$ have canonical coordinates and are expansive, to prove  Shadowing Lemma for Anosov families.

\subsection{Expansiveness}

The notion of expansiveness was introduced by Bowen \cite{Bowen2} and a sequence of studies occurred relating expansivity and hyperbolicity. Here, we will appropriately define expansiveness for non-stationary dynamical systems.

 \begin{definition} A non-autonomous dynamical system $\mathcal{F}  $  is \emph{expansive} if there exists $\delta>0$ such that for any distinct points $p,q\in \mathcal{M}$ there is some $n \in\mathbb{Z}$ such that $$d( \mathcal{F}^{n}_{0} (p), \mathcal{F}^{n}_{0} (q)) \geq \delta.$$
 \end{definition}

In \cite{JeoSS}, Proposition 7.5, is proved the following result:

\begin{proposition}\label{wer} For any $\mathcal{F}\in\mathcal{A}^{2}_{b}(\mathcal{M})$, there exist $r>0$ small enough, $\eta>0, \zeta>0$, with  $\eta^{-1}-\zeta>0$,  such that if  $p,q\in M_{0}$ and $d(\mathcal{F}_{0}^{n}(p),\mathcal{F}_{0}^{n}(q))<r$ for each $n\in [-N,N]$ for some $N\in\mathbb{N}$, then \[d(q ,p)\leq 2\sqrt{2}(\eta^{-1}-\zeta)^{-N}r.\] \end{proposition}

As a consequence we have:

  \begin{theoremb}\label{Expansive} Any $\mathcal{F} \in\mathcal{A}^{2}_{b}(\mathcal{M})$ is expansive.
\end{theoremb}
\begin{proof}
Take $r>0$, $\eta>0$ and $  \zeta>0$  as in Proposition \ref{wer}. Thus, if $p,q\in M_{0}$ and $d(f_{0}^{n}(p),f_{0}^{n}(q))<r$ for each $n\in \mathbb{Z}$, then $p=q$, which proves the theorem.
\end{proof}

\subsection{Shadowing Lemma}

In this subsection we use  canonical coordinates  and the expansiveness of Anosov families, proved above and in Section \ref{SectionCanonical},  to prove that elements in $\mathcal{A}^{2}_{b}(\mathcal{M})$ satisfies the shadowing property (see Definition \ref{Shadowing}).

\medskip

First let us remember the elements we are working on.
Consider  $\mathcal{F}=(f_{i})_{i\in\mathbb{Z}}\in\mathcal{A}^{2}_{b}(\mathcal{M})$.
Then  $\mathcal{F}$ admits canonical coordinates (Theorem A),  and  $\mathcal{F}$ is expansive (Theorem B).
 Take $\epsilon >0$ and $\delta>0$ as in Theorem A. For $p,q\in\mathcal{M}$ with $d(p,q)<\delta$  set \[[p,q] =\mathcal{W}^{s}(p,\epsilon)\cap \mathcal{W}^{u}(q,\epsilon)\quad \text{ and }\quad \mathcal{U}_{\delta} = \left\{(p,q) \in \coprod_{i\in \mathbb{Z}}M_{i} \times M_{i}: d(p,q) <\delta\right\}.\]

Then $[\cdot, \cdot ]:\mathcal{U}_{\delta}\longrightarrow \mathcal{M}$ is continuous,  because  $\mathcal{W}^{s}(p,\epsilon)$ and $\mathcal{W}^{u}(p,\epsilon)$ vary continuously with $p$.

 \begin{definition}\label{Shadowing} Given $\alpha>0$. A sequence $(x_{n})_{n\in\mathbb{Z}}$, where $x_{n}\in M_{n}$ for each $n\in\mathbb{Z}$, is  an $\alpha$-\emph{pseudo orbit} for   $\mathcal{F}$ if \[ d(f_{n}(x_{n}),x_{n+1})<\alpha\text{\quad for all }n\in\mathbb{Z}.\]
A pseudo orbit $(x_{n})_{n\in\mathbb{Z}}$ is $\epsilon$-\emph{shadowed} if there exists $y\in M_{0}$ such that  \[ d(\mathcal{F}_{0}^{n}(y),x_{n})<\epsilon\text{\quad for each }n\in\mathbb{Z}.\]
 \end{definition}

\medskip

Note that the first point of the pseudo orbit, and the point of the orbit which shadows the pseudo orbit, do not necessarily have to be in $M_0$, as in the case of the  finite pseudo-orbits.

Now we state a version of the Shadowing Lemma for Anosov family.

\begin{theoremc}\emph{(Shadowing Lemma for Anosov Family)} Let   $\mathcal{F}\in\mathcal{A}^{2}_{b}(\mathcal{M})$.
Given $\beta>0$ there is an $\alpha>0$ such that every $\alpha$-pseudo orbit $(x_n)_{n \in \mathbb{Z}}$ is $\beta$-shadowed for an unique orbit of $\mathcal{F}$ through  a $y \in \mathcal{M}$.
\end{theoremc}

To prove this result we use the same idea as in the case of Anosov diffeomorphisms, but respecting this class of systems.

\begin{proof} (of Theorem C)

Given $\beta>0$. In order to simplify the demonstration we will divide the proof into some  steps. In the first step we define parameters. In the second step we define the appropriate $\alpha$ for the pseudo orbit. In the third step we prove that all types of pseudo orbit are shadowed. And in the last step  we prove the uniqueness of the orbit  which shadows the pseudo orbit.

\medskip

\noindent (i) Choice of parameters:
\medskip

Choose $\epsilon >0$ as in   Theorem \ref{variedadeestavel}. This ensures that $\mathcal{W}^s(x,\epsilon)$ and $\mathcal{W}^u(x,\epsilon)$ are  disks that varies continuously with $x$.  
Consider the parameters:

\noindent -$\lambda \in (0,1)$ the hyperbolic constant of $\mathcal{F}$.\\
-$\epsilon_1 <$ (1- $\lambda$) min $\{\epsilon, \beta\}$.\\
-$\eta = \dfrac{\epsilon_1}{1-\lambda}$  (note that $\eta < \epsilon$ and $\eta< \beta$).\\
- $\delta < \beta - \eta$ positive constant for which  $[\cdot , \cdot ]_{\epsilon_{1},\delta} : \mathcal{U}_{\delta} \longrightarrow \mathcal{M}$ is well defined, that is, \begin{center}if $d(x,y)<\delta$  then $\mathcal{W}^s(x,\epsilon_1) \cap \mathcal{W}^u(y,\epsilon_1) = [x,y]. $\end{center}

 \medskip

\noindent (ii) Now we choose the  $\alpha >0$ (for the  pseudo orbit)  appropriately:

\medskip

Since $[\cdot, \cdot ] $ is continuous and the stable and unstable subspaces are orthogonal (see proof of Theorem A, Figure \ref{figure1}), we can find an $\alpha>0$ such that if $d(z,w)<\alpha$  then
\[  \mathcal{W}^{s}(z,\epsilon_{1}) \cap \mathcal{W}^u(x,\epsilon_{1}) \, \in \,\mathcal{W}^{s}(z,\epsilon_{1}),\quad \text{ for any  }x \in \mathcal{W}^{s} (w,\lambda\epsilon_{1}).\]
Hence
  $$[ z,\mathcal{W}^{s}(w,\lambda \epsilon_{1})]: =\{ [z,x]:x \in \mathcal{W}^{s} (w,\lambda\epsilon_{1})\} \subseteq \mathcal{W}^{s}(z,\epsilon_{1}).$$

\medskip

\noindent (iii) We divide the possible  types of pseudo orbit in  three cases  and we show that in any of them the $\alpha$-pseudo orbit is shadowed.

 Suppose that we have a finite $\alpha$-pseudo orbit  $\underline{x} = [x_0, x_1, \cdots, x_n ] $,  where $x_{i}\in M_{i}$, with  $i\in \{0,\dots,n\}$.  We prove that
\begin{equation*}y_0= x_0, \quad
y_1=[x_1,f_0(y_0)],\quad
y_2 = [x_2, f_1(y_1)],
\quad\dots \quad
y_n = [ x_n, f_{n-1}(y_{n-1})].\end{equation*}
  is a sequence well defined
  and  it  shadows the $\alpha$-pseudo orbit $\underline{x}$. In order to prove this fact, we set  recursively $y_k = [ x_k, f_{k-1}(y_{k-1})]$.  Suppose that   $y_{0},\dots, y_{k}$ are well defined, for any $k<n.$ Since
 $y_k \in \mathcal{W}^{s} (x_k,\epsilon_{1})$ we have $f_k(y_k) \in \mathcal{W}^{s} (f_k(x_k),\lambda\epsilon_{1})$.  Let us remember that $\lambda$ is the hyperbolicity constant.
 Thus,
$$d(x_{k+1}, f_k(x_k))< \alpha \text{ implies  }y_{k+1} = [ x_{k+1}, f_k(y_k)] \in [ x_{k+1}, \mathcal{W}^{s}(f_k(x_k), \lambda\epsilon_{1})] \subseteq \mathcal{W}^{s}(x_{k+1}, \epsilon_{1}).$$
So $y_{k+1}$ is well defined.

Next, we know that $y_k \in \mathcal{W}^{u}(f_{{k-1}}(y_{k-1}),{\epsilon_1}) $ implies $ f_{k-1}^{-1}(y_k) \in \mathcal{W}^{u}(y_{k-1},{\lambda \epsilon_1})$.
Recursively, $\mathcal{F}_{k-j}^{-j}(y_k) \in \mathcal{W}^{u}(y_{k-j},{\theta_j})$ where $\theta_j = \sum_{i=1}^{j} \lambda^{j}{\epsilon_1} < \eta $. Consider \[y= f_{0}^{-1}\circ f_{1}^{-1} \circ \cdots \circ f_{n-1}^{-1}(y_{n}).\]
Note that $$ \mathcal{F}_0^{j}(y) = \mathcal{F}_0^{(-n+j)}(y_n) \in \mathcal{W}^{u}(y_{n-(n-j)},\theta_{n-j})=\mathcal{W}^u (y_j, {\theta_{n-j}}), $$    where   $$\theta_{n-j} = \sum_{i=1}^{n-j}\lambda^{i}\epsilon_1 < \eta.$$  Thus,
\begin{center}$d( \mathcal{F}_0^j(y),x_j) \leq d( \mathcal{F}_0^j(y),y_j)+ d(y_j,x_j) \leq \eta + \delta < \eta + \beta - \eta = \beta.$\end{center}

Therefore, we  conclude that $\underline{x}$ is $\beta$-shadowed by the orbit   of $y= \mathcal{F}_n^{-n}(y_n)$.

\vspace{0.3cm}

The second case, we suppose that  $ \underline{x} = [x_{-n}, \cdots, x_0, \cdots, x_n ]$,  where $x_{i}\in M_{i}$, for $i\in \{-n,\dots,n\}$.

\medskip

Consider the reorganized sequence $\underline{\tilde{x}} = [ \tilde{x_0},\tilde{x_1}, \cdots, \tilde{x}_{2n}]$, where  $\tilde{x_i}=x_{-n+i}$, with $i=0,1,\dots,2n$.
As we show in (iii)-first case, $\underline{\tilde{x}}$\;  is shadowed by $\tilde{y} \in M_{-n}$. So,
\[d( \mathcal{F}_{-n}^{j}(\tilde{y}),\tilde{x_j}) = d( \mathcal{F}_{-n}^{j}(\tilde{y}), x_{-n+j})<\beta,  \text{ for } j = \{ 0, 1, \cdots, 2n\}.\] Therefore, if $y=\mathcal{F}_{-n}^{n}(\tilde{y})$, we have   $d( \mathcal{F}_0^{j}( y), x_{j})<\beta$, for  $j = \{-n, \cdots, 0, \cdots, n\}$.

\vspace{0.3cm}

 Finally, the last case, we consider the infinite $\alpha$-pseudo orbit   $\underline{x}= [ \dots, x_{-n}, \dots, x_0, \dots, x_n, \dots].$

\medskip


For each $n>0$ consider $\underline{x_n} = [ x_{-n}, \cdots, x_0, \cdots, x_n]$ and $y_n \in M_{0}$ its shadow.  As $M_{0}$ is compact, there exists a subsequence $n_{k}$ in $ \mathbb{N}$  and  $y \in M_{0}$ such that $\lim_{k\rightarrow \infty}{y_{n_k}} = y$.  Since $d(\mathcal{F}_0^{j}(y_{n}), x_j) < \beta$ for all $j \in \{-n, \dots, n\}$ and $n\in\mathbb{N}$, when $k\rightarrow \infty$ we have $d(\mathcal{F}_0^{j}(y),x_j)<\beta$, for all $j \in \mathbb{Z}$.

\medskip

\noindent (vi) Uniqueness of the orbit which shadows $\underline{x}= [ \dots, x_{-n}, \dots, x_0, \dots, x_n, \dots]$ follows from the expansiveness of $ \mathcal{F}$.
\end{proof}



\section{Markov partition for Anosov family}

In this section we will prove Theorem D. 
To prove the theorem we will use  ideas of  \cite{Aoki} (Chapter 4) and \cite{Z03} (Theorem 3.12), whose authors used them to prove the existence of Markov partitions for TA-homeomorphism (see the definition in \cite{Aoki}, Chapter 1) and  Axiom A diffeomorphisms, respectively. The proofs of some of the results that will be presented here can be done as in the singular case and therefore we will omit them.

\medskip

According to definition of  Markov Partition for Anosov Families (Definition \ref{Markov}) we will need to consider  for each  $i\in\mathbb{Z}$,   $M_{i}=M\times \{i\}$ for a fixed compact Riemannian manifold $M$. We will comment on this below.

 \begin{definition}A subset $R\subseteq M_{i}$ is called a \emph{rectangle} if,
      for any $x,y\in R$, $[x,y]$ is defined  and
     belongs to  $R$. We say that  $R$ is   \emph{proper} if $R=\overline{\text{int } {R}}.$\end{definition}

For $x\in R$ and $\varepsilon>0$ small enough, set \begin{enumerate}[(i)] \item $  \mathcal{W}^{s}(x,R)=\mathcal{W}^{s}(x,\varepsilon)\cap R ,$
\item $ \partial ^{s}R=\{x\in R: x\notin \text{int}( \mathcal{W}^{u}(x, R))\}$,
\item $  \partial ^{u}R=\{x\in R: x\notin \text{int}( \mathcal{W}^{s}(x, R))\}$, \end{enumerate}
where the interior of $  \mathcal{W}^{u}(x, R)$ and $  \mathcal{W}^{s}(x, R)$ are taken as subsets of $\mathcal{W}^{u}(x,\varepsilon)$ and $\mathcal{W}^{s}(x,\varepsilon)$, respectively.

\medskip

 The following property can be proved as in the single case (see  \cite{Z03}).

     \begin{lemma}\label{lema63}
  Let $R$ and $T$ be rectangles. Thus:
  \begin{enumerate}[(i)]
      \item $\overline{R}$ is a rectangle.
      \item If $\text{int}(R) \neq \emptyset,$ then $\text{int}(R) $ is a rectangle.
      \item If $R\cap T\neq \emptyset$, then $R\cap T$ is a rectangle.
      \item
 If $R$ is a closed rectangle, then $\partial R=\partial ^{s}R\cap \partial ^{u}R.$
  \end{enumerate}
\end{lemma}

\begin{definition}\label{Markov} For  $\mathcal{F}=(f_{i})_{i\in\mathbb{Z}}$, a \emph{Markov partition} is a sequence of finite
partitions $$\mathcal{R}^{i}=\{R_{1}^{i}, R_{2}^{i},\dots, R_{n_{i}}^{i}\}$$  of $M_{i}$, i.e. coverings of $M_{i}$ by closed sets with disjoint interiors, such that $\max_{i}\text{Card}(\mathcal{R}^{i})<\infty$,
each partition element is a proper rectangle, and satisfies the \textit{Markov condition}:
for $R_{j}^{i}\in  \mathcal{R}_{i}$ and $R_{k}^{i+1}\in  \mathcal{R}_{i+1}$, if  $x \in R_{j}^{i}$ and $f_{i}(x) \in  R_{k}^{i+1}$, then
\[ \mathcal{W}^{u}(f_{i}(x), R_{k}^{i+1})\subseteq f_{i}(\mathcal{W}^{u}(x, R_{j}^{i}))
\quad \text{and}\quad
f_{i}(\mathcal{W}^{s} (x, R_{j}^{i})) \subseteq \mathcal{W}^{s}(f_{i}(x), R_{k}^{i+1} ).\]\end{definition}

\medskip

Consider $\mathcal{F}\in\mathcal{A}_{b}^{2}(\mathcal{M})$.
Let $\beta>0$ be very small and choose $\alpha>0$ small as in Theorem C, that is, every $\alpha$-pseudo-orbit in ${M}$ is $\beta$-shadowed by a orbit through a unique point in ${M}$. Since $\mathcal{F}\in\mathcal{A}_{b}^{2}(\mathcal{M})$, we can choose $\gamma\in (0,\min\{\beta,\alpha/2\})$ such that \begin{equation}\label{eqi1} d(f_{n}(x),f_{n}(y))<\alpha/2, \text{ when }d(x,y)<\gamma.\end{equation}

In order to satisfy the condition  $\max_{i}\text{Card}(\mathcal{R}^{i})<\infty$ in the Definition \ref{Markov},  we will suppose  that,  for each  $i\in\mathbb{Z}$,   $M_{i}=M\times \{i\}$ for a fixed compact Riemannian manifold $M$\footnote{If each $M_{i}$ is a different manifold, the set the cardinality of the sequence $P_{i}$ could be not bounded.}. We prove that in this case,

\begin{theoremd}  Suppose    that  $M_{i}=M\times \{i\}$ for each $i\in\mathbb{Z},$ where $M$ is a fixed compact Riemannian manifold. Each $\mathcal{F}\in\mathcal{A}_{b}^{2}(\mathcal{M})$ admits a Markov partition.
\end{theoremd}

In order to prove Theorem D, first we  prove a series of lemmas.

\medskip

Let $P=\{p_{1},\dots,p_{r}\}$ be a $\gamma$-dense subset of $M$. Hence $P_{i}=\{(p_{1},i),\dots,(p_{r},i)\}$ is a $\gamma$-dense subset of $M_{i}$. To simplify the notation, we will write $p_{j}$ instead of $(p_{j},i)$ for each $i\in\mathbb{Z}$, $j=1,\dots,r$. Set
$$\Sigma_{0}(P)=\left\{\bar{a}=(\dots,a_{-1},a_{0},a_{1},\dots), a_i \in M_i, \text{ and }   \bar{a} \in \prod_{-\infty}^{\infty}P: d(f_{n}(a_{n}),a_{n+1})<\alpha\text{ for all }n\right\}.$$
 That is, $ \Sigma_{0}(P)$ is a set consisting of   $\alpha$-pseudo orbit of $\mathcal{F}$. $\Sigma_{0}(P)$ will be  endowed with the compact topology.
It follows from Theorem C that for each $\bar{a}\in \Sigma_{0}(P)$ there is a unique $\theta_{0}(\bar{a})\in M_{0}$ which $\beta$-shadows  the $\alpha$-pseudo orbit $ \bar{a}$.

\begin{lemma}     $\theta_{0}: \Sigma_{0}(P)\rightarrow M_{0}$ is continuous.\end{lemma}
\begin{proof} Suppose that   $\theta_{0}$ is not continuous. Thus,  there is a $\gamma>0$ such  that for every $n\in\mathbb{N}$ we can find $\bar{a}_{n}, \bar{b}_{n} \in \Sigma_{0}(P)$, with $a_{n,j}=b_{n,j}$ for all $j\in [-n,n]$,  but $d(\theta_{0} (\bar{a}_{n}),\theta_{0}(\bar{b}_{n}))\geq \gamma.$  Therefore, for all $j\in [-n, n]$, we have   \begin{align*} d(\mathcal{F}_{0}^{j}(\theta_{0}(\bar{a}_{n})),\mathcal{F}_{0}^{j}(\theta_{0}(\bar{b}_{n})))\leq d(\mathcal{F}_{0}^{j}(\theta_{0}(\bar{a}_{n})),a_{n,j+1})+d(b_{n,j+1},\mathcal{F}_{0}^{j}(\theta_{0}(\bar{b}_{n}))) \leq 2\beta .\end{align*}
We may assume $\theta(\bar{a}_{n})\rightarrow a$ and $\theta(\bar{b}_{n})\rightarrow b$ as $n\rightarrow \infty$. Hence $d(\mathcal{F}^{j}_{0}(a),\mathcal{F}^{j}_{0}(b))\leq 2\beta $ for all $j\in\mathbb{Z}$ and $d(a,b
)\geq \gamma$, which contradicts the expansiveness of $\mathcal{F}$.\end{proof}

 \begin{lemma}     $\theta_{0}: \Sigma_{0}(P)\rightarrow M_{0}$ is surjective. \end{lemma}
 \begin{proof}Fix $x_{0}\in M_{0}$ and set $x_{n}=\mathcal{F}^{n}(x_{0})$ for $n\in\mathbb{Z}$. Since $P$ is a $\gamma$-dense subset, there exists a $a_{n}\in P$ such that $d(x_{n},a_{n})<\gamma$ for each $n\in\mathbb{Z}$. By \eqref{eqi1} we have  $d(f_{n}(x_{n}),f_{n}(a_{n}))<\alpha/2$ for each $n\in\mathbb{Z}$. Therefore \[ d(f_{n}(a_{n}),a_{n+1})\leq d(f_{n}(a_{n}),f_{n}(x_{n}))+d( f_{n}(x_{n}), a_{n+1})<\alpha/2+d(  x_{n+1}, a_{n+1}) <\alpha. \]
  Consequently, $\bar{a}=(a_{n})_{n\in\mathbb{Z}}\in \Sigma_{0}(P)$ and
   $\theta_{0}(\bar{a})=x,$ which proves that $\theta_{0}$ is surjective.\end{proof}

For each $\bar{a}=(a_{n})_{n\in\mathbb{Z}}\in \Sigma_{0}(P)$ and  $i\in \mathbb{Z}$, set \[\sigma(\bar{a})=(a_{n+1})_{n\in\mathbb{Z}},\quad  \Sigma_{i}(P)=\sigma^{i}(\Sigma_{0}(P))\quad \text{
and take}\quad \sigma_{i}:=\sigma|_{\Sigma_{i}(P)}:\Sigma_{i}(P)\rightarrow \Sigma_{i+1}(P).\]

Let $\theta_{i}: \Sigma_{i}(P)\rightarrow M_{i}$ be inductively defined such that the following diagram commutes:
\[\begin{CD}\Sigma_{-1}(P)@>{\sigma_{-1}}>> \Sigma_{0}(P)@>{\sigma_{0}}>>\Sigma_{1}(P)@>{\sigma_{1}}>>\Sigma_{2}(P) \\ @V{\cdots}V{\theta_{-1}}V @VV{\theta_{0}}V @VV{\theta_{1}}V @VV{\theta_{2}\text{ }\cdots}V\\ M_{-1}@>{f_{-1}}>> M_{0}@>{f_{0}}>>M_{1} @>{f_{1}}>>M_{2} \end{CD}\]
  that is, \(\theta_{i+1}(\bar{a})=f_{i}(\theta_{i}(\sigma_{i}^{-1}(\bar{a})))\) for each \(\bar{a}\in \Sigma_{i}(P).\)  Since \(\theta_{0}: \Sigma_{0}(P)\rightarrow M_{0}\) is continuous and surjective, $\theta_{i}: \Sigma_{i}(P)\rightarrow M_{i}$ is continuous and  surjective.

\medskip

Fix $i\in\mathbb{Z}$. For $\bar{a},\bar{b} \in \Sigma_{i}(P)$ with $a_{0}=b_{0} $ we define $ [ \bar{a},\bar{b}]^{i}\in \Sigma_{i}(P)$ by
\begin{equation*}
 [ \bar{a},\bar{b}]^{i}_{j} =
\begin{cases}
        a_{j}  & \mbox{for  } j\geq 0 \\
        	b_{j}  & \mbox{for  } j\leq 0 .
        \end{cases}
\end{equation*}
If $\bar{c} =[ \bar{a},\bar{b}]^{i}$, we have  \[d(\mathcal{F}_{i}^{j}(\theta_{i}(\bar{c} )),\mathcal{F}_{i}^{j}(\theta_{i}(\bar{a} )))\leq 2\beta \text{ for }j\geq 0\quad\text{and}\quad d(\mathcal{F}_{i}^{j}(\theta_{i}(\bar{c} )),\mathcal{F}_{i}^{j}(\theta_{i}(\bar{b}  )))\leq 2\beta \text{ for }j\leq 0.\]
It follows from Proposition \ref{dasddd} that $\theta_{i}(\bar{c})\in \mathcal{W}^{s}(\theta_{i}(\bar{a}),2\beta)\cap \mathcal{W} ^{u}(\theta_{i}(\bar{b}), 2\beta)=[\theta_{i}(\bar{a}),\theta_{i}(\bar{b})]$. This fact proves that   \begin{equation}\label{eefe} \theta_{i}([ \bar{a},\bar{b}]^{i}) =[\theta_{i}(\bar{a}),\theta_{i}(\bar{b})].\end{equation}

 For each $i\in\mathbb{Z}$ and $k=1,\dots,r$, set $$T_{k}^{i}=\{ \theta_{i}(\bar{a}):\bar{a}=(\dots, a_{-1},a_{0},a_{1},\dots)\in \Sigma_{i}(P),a_{0}=p_{k}\}.$$

 \begin{lemma}\label{primlemma} $T_{k}^{i}$ is a rectangle. \end{lemma}
\begin{proof}We   prove  that if  $x,y\in T_{k}^{i}$ then  $[x,y] \in T_{k}^{i}.$ Take $x=\theta_{i}(\bar{a}),$ $y=\theta_{i}(\bar{b})\in T_{k}^{i}$ (thus $a_{0}=p_{k}=b_{0}$).   If  $\bar{c}=[\bar{a},\bar{b}]^{i}$, then $c_{0}=p_{k}$, that is, $\theta_{i}(\bar{c})\in T_{k}^{i}$.   It follows from \eqref{eefe} that $[x,y]=\theta_{i} (\bar{c})\in T_{k}^{i}.$\end{proof}

\begin{lemma} If  $x \in T_{j}^{i}$ and $f_{i}(x) \in  T_{k}^{i+1}$, then
\[\mathcal{W}^{u}(f_{i}(x), T_{k}^{i+1})\subseteq f_{i}(\mathcal{W}^{u}(x, T_{j}^{i}))
\quad \text{and}\quad
f_{i}(\mathcal{W}^{s} (x, T_{j}^{i})) \subseteq \mathcal{W}^{s}(f_{i}(x), T_{k}^{i+1} ).\]\end{lemma}
\begin{proof}
Since $x \in T_{j}^{i}$ and $f_{i}(x) \in  T_{k}^{i+1}$, we have   $x=\theta_{i}(\bar{a})$ with $a_{0}=p_{j}$ and $a_{1}=p_{k}$, because $f_{i}(\theta_{i}( \bar{a}))=\theta_{i+1}(\sigma_{i}(\bar{a}))\in T_{k}^{i+1}$. Take  $y\in \mathcal{W}^{s}(x,T_{j}^{i})=\mathcal{W}^{s}(x,\varepsilon)
\cap T_{j}^{i}$. Then, we can write  $y=\theta_{i}(\bar{b})$, with  $b_{0}=p_{j}$.
Therefore \[y=[x,y]=\theta_{i}([\bar{a},\bar{b}]^{i})\quad \text{ and thus}\quad f_{i}(y)=f_{i}(\theta_{i}(\bar{b}))=\theta_{i+1}\sigma_{i}  ( [\bar{a},\bar{b}]^{i})\in T_{k}^{i+1},\]
because     $a_{1}=p_{k}$. Since $y\in \mathcal{W}^{s}(x,\varepsilon)$, we have $ f_{i}(y)\in \mathcal{W}^{s}(f_{i}(x),\varepsilon)$. Therefore,    $f_{i}(y)\in   \mathcal{W}^{s}(f_{i}(x),T_{k}^{i+1})$. We have proved
 \begin{equation}\label{erwer123}
     f_{i}(\mathcal{W}^{s}(x,T_{j}^{i}))\subseteq   \mathcal{W}^{s}(f_{i}(x),T_{k}^{i+1}).
 \end{equation}
 Analogously we can prove  
 \begin{equation}\label{erwer1232}  \mathcal{W}^{u}(f_{i} (x),T_{k}^{i+1})\subseteq f_{i}(\mathcal{W}^{u}(x,T_{j}^{i})),\end{equation} which proves the lemma. \end{proof}

\begin{lemma}   $T_{k}^{i}$ is closed and $J^{i}=\{T_{1}^{i},\dots T_{r}^{i}\}$ is a covering  of $M_{i}$.\end{lemma}

\begin{proof} Since $T_{k}^{i}=\theta_{i}(\Pi_{k}^{i})$, where $\Pi_{k}^{i}=\{\bar{a}\in \Sigma_{i}(P): a_{0}=p_{k}\}$  is a closed subset of $\Sigma_{i}(P)$ and $\theta_{i}$ is continuous, we have   $T_{k}^{i}$ is closed  (note that $\Sigma_{i}(P)$ is compact). Furthermore, given that $\theta_{i}$ is surjective and $\Sigma_{i}(P)= \underset{j=1,\dots,r} \bigcup   \Pi_{j}^{i}$, we have  $J^{i}=\{T_{1}^{i},\dots T_{r}^{i}\}$ is a covering  of $M_{i}$. \end{proof}

Next, we  will build a first refinement of $J^{i}$,
since the interiors of the rectangles above could intersect. For $T_{j}^{i}\cap T_{k}^{i}\neq \emptyset$, let
\begin{align*}T_{j,k}^{i,1}&=\{x\in T_{j}^{i}:  \mathcal{W}^{u}(x,T_{j}^{i})\cap T_{k}^{i}\neq
\emptyset,  \mathcal{W}^{s}(x,T_{j}^{i})\cap T_{k}^{i}\neq
\emptyset \}=T_{j}^{i}\cap T_{k}^{i}\\
T_{j,k}^{i,2}&=\{x\in T_{j}^{i}:\mathcal{W}^{u}(x,T_{j}^{i})\cap T_{k}^{i}\neq  \emptyset , \mathcal{W}^{s}(x,T_{j}^{i})\cap T_{k}^{i}=  \emptyset\}\\
T_{j,k}^{i,3}&=\{x\in T_{j}^{i}: \mathcal{W}^{u}(x,T_{j}^{i})\cap T_{k}^{i}=\emptyset, \mathcal{W}^{s}(x,T_{j}^{i})\cap T_{k}^{i}\neq  \emptyset \} \\
T_{j,k}^{i,4}&=\{x\in T_{j}^{i}: \mathcal{W}^{u}(x,T_{j}^{i})\cap T_{k}^{i}= \emptyset, \mathcal{W}^{u}(x,T_{j}^{i})\cap T_{k}^{i} =  \emptyset \}.
 \end{align*}
\begin{lemma} For $n=1,2,3,4$,  $T^{i,n}_{j,k}$ is a rectangle.\end{lemma}
 \begin{proof} Fix $x,y\in T_{j,k}^{i,n}$. Thus $x,y\in T_{j}^{i}$ and therefore $[x,y]\in T_{j}^{i}$ (Lemma \ref{primlemma}).  Given that $[x,y]\in \mathcal{W}^{s}(x,\varepsilon),$ then  \[\mathcal{W}^{s}([x,y],T_{j}^{i})=\mathcal{W}^{s}([x,y],\varepsilon)\cap T_{j}^{i}=\mathcal{W}^{s}(x,\varepsilon) \cap T_{j}^{i}=\mathcal{W}^{s}(x,T_{j}^{i})\]
   and since $[x,y]\in \mathcal{W}^{u}(y,\varepsilon),$ then  \[\mathcal{W}^{u}([x,y],T_{j}^{i})=\mathcal{W}^{u}([x,y],\varepsilon)\cap T_{j}^{i}=\mathcal{W}^{u}(y,\varepsilon) \cap T_{j}^{i}=\mathcal{W}^{u}(y,T_{j}^{i}).\]
   These facts  imply that $[x,y] \in  T^{i,n}_{j,k}$ and hence  $T^{i,n}_{j,k}$ is a rectangle for $n=1,2,3,4$ (see \cite{Aoki}, Remark 4.2.3, for more detail in the single case, which work for families).\end{proof}

For each $x\in M_{i}$, set
\begin{align*} J^{i}(x)&=\{T_{j}^{i}\in J_{i}:x\in T_{j}^{i}\}\\
 J^{i}_{\ast}(x)&=\{T_{k}^{i}\in J^{i}: T_{k}^{i}\cap T_{j}^{i}\neq \emptyset\text{ for some }T_{j}^{i}\in J^{i}(x)\}\\
 Z^{i}&=M_{i}\setminus  \cup_{j=1}^{r} \partial T_{j}^{i}\\
   Z^{i}_{\ast}&=\{ x\in M_{i}: \mathcal{W}^{s}(x,\varepsilon) \cap  \partial ^{s}T_{k}^{i}=\emptyset \text{ and } \mathcal{W}^{u}(x,\varepsilon)\cap  \partial^{u}  T_{k}^{i}=\emptyset \text{ for all } T_{k}^{i}\in J^{i}_{\ast}(x)\}.
\end{align*}

Since $J^{i}$ is a closed cover of $ M_{i}$, we have  $Z^{i}$ is an open dense subset of $M_{i}$.
 The proof for single maps works to prove that   $Z^{i}_{\ast}$  is open and dense in $M_{i}$ (see \cite{Aoki}), Lemma 4.2.1). Furthermore,   each $x\in   Z_{\ast}^{i}$ lies in $\text{int} (T^{i,n}_{j,k})$ for some $n$ (see \cite{Aoki}, Remark 4.2.5).

\medskip

For $x\in Z_{\ast}^{i}$ define \[R^{i}(x)=\bigcap \left\{\text{int}( T^{i,n}_{j,k}):  x\in T_{j}^{i},T_{k}^{i}\cap T_{j}^{i}\neq \emptyset\text{  and } x\in T^{i,n}_{j,k}\right\}.\]
By Lemma \ref{lema63} we have $R^{i}(x)$ is an open rectangle ($R^{i}(x)$ is a finite intersection of open subsets). Consequently, $\overline{R^{i}(x)}$ is proper.

\begin{lemma} For any  $y\in R^{i}(x)\cap Z^{i}_{\ast}$, we have $J^{i}(x)= J^{i}(y)$ and $R^{i}(y)=R^{i}(x)$. \end{lemma}
\begin{proof}See \cite{Aoki}, Remark 4.2.6.
\end{proof}


Therefore, there are only finitely many distinct $R^{i}(x)$'s. Let
 \[\mathcal{R}^{i}=\{  \overline{R^{i}(x)}: x\in Z_{\ast}^{i}\}=\{R^{i}_{1},\dots,R^{i}_{m_{i}}\}\quad \text{for }i\in\mathbb{Z}. \]

Finally we prove that:

  \begin{theorem} \label{Final} The sequence $\mathcal{R}^{i}$ for $i\in\mathbb{Z}$ is a Markov partition for $\mathcal{F}$.
 \end{theorem}

\begin{proof}
We obtained that if $z\in Z_{\ast}^{i}$, then $R^{i}(z)=R^{i}(x)$ or $R^{i}(z)\cap R^{i}(x)=\emptyset$. Therefore  \[(\overline{R^{i}(x)}\setminus R^{i}(x))\cap Z_{\ast}^{i}=\emptyset .\]
Since $Z_{\ast}^{i}$ is dense in $M_{i}$, we have  $\overline{R^{i}(x)}\setminus R^{i}(x)$ has no interior in $M_{i}$ and $R^{i}(x)=\text{int}(\overline{R^{i}(x)})$. Therefore, for $R^{i}(x)\neq R^{i}(z)$, we have
 \[ \text{int}(\overline{R^{i}(x)})\cap \text{int}(\overline{R^{i}(z)}) =R^{i}(x)\cap R^{i}(z)=\emptyset. \]

\noindent \textbf{Claim:} If $x,y\in Z_{\ast}^{i}\cap f_{i}^{-1}(Z_{\ast}^{i+1})$, $R^{i}(x)=R^{i}(y)$ and $y\in \mathcal{W}^{s}(x,\varepsilon)$, then
\begin{enumerate}[i)]
\item $J^{i+1}(f_{i}(x))=J^{i+1}(f_{i}(y))$.
\item $R^{i+1}(f_{i}(x))=R^{i+1}(f_{i}(y))$.
\end{enumerate}

\begin{proof}(of Claim)

For i),  assume  that  $f_{i}(x)=\theta_{i+1}(\sigma_{i}(\bar{a}))\in T_{j}^{i+1}$ ($a_{1}=p_{j}$) and  $a_{0}=p_{s}$  (that is, $x=\theta_{i} (\bar{a})\in T_{s}^{i}$). By \eqref{erwer123} we have
$$f_{i}(y)\in f_{i}(\mathcal{W}^{s}(x,T_{s}^{i}))\subseteq  \mathcal{W}^{s}(f_{i}(x),T_{j}^{i+1}),$$
therefore $f_{i}(y)\in T_{j}^{i+1}
$. Similarly, if $f_{i}(y)\in T_{j}^{i+1}$, then $f_{i}(x)\in T_{j}^{i+1}$, and therefore $J^{i+1}(f_{i}(x))=J^{i+1}(f_{i}(y))$.

\medskip

For ii),  we prove that if $T_{j}^{i+1}\in J^{i+1}(f_{i}(x))=J^{i+1}(f_{i}(y))$ and $T_{k}^{i+1}\cap T_{j}^{i+1}\neq \emptyset$ for $T_{k}^{i+1}\in J^{i+1}$,
then  $f_{i}(x),f_{i}(y)$ belong to the same $T^{n,i+1}_{j,k}$.  Since $f_{i}(y)\in \mathcal{W}^{s} (f_{i}(x),{\varepsilon})$, we have $\mathcal{W}^{s}(f_{i}(y),T_{j}^{i+1})=\mathcal{W}^{s}(f_{i}(x),T_{j}^{i+1})$. Thus $f_{i}(x)$ and $f_{i}(y)$ belong to   $T^{1,i+1}_{j,k}\cup T^{3,i+1}_{j,k}$ or $T^{2,i+1}_{j,k}\cup T^{4,i+1}_{j,k}$. Suppose  $$\mathcal{W}^{u}(f_{i}(y),T_{j}^{i+1})\cap T_{k}^{i+1}=\emptyset\quad\text{and}\quad    \mathcal{W}^{u}(f_{i}(x),T_{j}^{i+1})\cap T_{k}^{i+1}\neq \emptyset.$$
Take $f_{i}(z)\in  \mathcal{W}^{u}(f_{i}(x),T_{j}^{i+1})\cap T_{k}^{i+1}$.  From \eqref{erwer1232} we have
 $ f_{i}(z)\in   f_{i}(\mathcal{W}^{u}(x,T_{s}^{i}))$, that is,
   $z\in \mathcal{W}^{u}(x,T_{s}^{i})$, since $x\in T_{s}^{i}$. Write  $f_{i}(z)=\theta_{i+1}(\sigma_{i}(\bar{b}))$, where $a_{1}=p_{k}$ and $a_{0}=p_{t}$ for some $t=1,\dots,k$.
 Then $z\in T_{t}^{i}$ and $f_{i}(\mathcal{W}^{s}(z,T_{t}^{i}))\subseteq \mathcal{W}^{s}(f_{i}(z),T_{k}^{i+1})$. Hence $z\in T_{t}^{i}\cap T_{s}^{i}\neq \emptyset$.
  Since $x\in T_{s}^{i}$, we have  $T_{s}^{i}\in J^{i}(x)=J^{i}(y)$.

 Now, given that   $z\in \mathcal{W}^{u}(x,T_{s}^{i})\cap T_{t}^{i}$ and   $x,y$ are in the same $T^{n,i}_{s,t}$, there exists some $w\in  \mathcal{W}^{u}(y,T_{s}^{i})\cap T_{t}^{i}$. Hence
 $$v=[z,y]=[z,w]\in \mathcal{W}^{s}(z,T_{t}^{i}
)\cap\mathcal{W}^{u}(y,T_{s}^{i}) $$
 and, since $f_{i}(z),f_{i}(y)\in T_{j}^{i+1}$ and $T_{j}^{i+1}$ is a rectangle, we have $$f_{i}(v)=[f_{i}(z),f_{i}(y)]\in \mathcal{W}^{u}(f_{i}(z),T_{k}^{i+1})\cap \mathcal{W}^{u}(f_{i}(y),T_{j}^{i+1}),$$ which is a contradiction. Therefore $R^{i+1}(f_{i}(x))=R^{i+1}(f_{i}(y))$.
\end{proof}

 The rest of the proof, which  we present below, is taken from the proof of Bowen   for Anosov diffeomorphisms (\cite{Z03}). All the facts are topological and are valid for our case.

 \medskip

 For small $\delta>0$,  set
 $$Y_{s}^{i}=\bigcup\left\{\mathcal{W}^{s}(z,\delta): z\in \underset{j}\cup\partial^{s}T_{j}^{i}\right\}\quad\text{and}\quad Y_{u}^{i}=\bigcup\left\{\mathcal{W}^{u}(z,\delta): z\in \underset{j}\cup\partial^{u}T_{j}^{i}\right\}.$$  $Y_{s}^{i}$ and $Y_{u}^{i}$ are closed and nowhere dense. Hence $M_{i}\setminus (Y_{s}^{i}\cup Y_{u}^{i})\subseteq Z_{\ast}^{i}$ is open and dense in $M_{i}$. Furthermore, if $x\notin (Y_{s}^{i}\cup Y_{u}^{i})\cap f_{i}^{-1}(Y_{s}^{i+1}\cup Y_{u}^{i+1})$, then $x\in Z_{\ast}^{i}\cap f^{-1}_{i}(Z_{\ast}^{i+1})$ and hence the set  $\{z\in \mathcal{W}^{s}(x,R^{i}(x)): z\in Z_{\ast}^{i}\cap f_{i}^{-1}(Z_{\ast}^{i+1})\}$ is open and dense in $\mathcal{W}^{s}(x,\overline{R^{i}(x)})$  (as a subset of $\mathcal{W}^{s}(x,\varepsilon)\cap M_{i}$).  By the previous claim we have $R^{i+1}(f_{i}(y))=R^{i+1}(f_{i}(x))$ for $y$ in $\{z\in \mathcal{W}^{s}(x,R^{i}(x)): z\in Z_{\ast}^{i}\cap f_{i}^{-1}(Z_{\ast}^{i+1})\}$.

 By continuity
 $$f_{i}(\mathcal{W}^{s}(x,\overline{R^{i}(x)}))\subseteq \overline{R^{i+1}(f_{i}(x))} $$
and since $f_{i}(\mathcal{W}^{s}(x,\overline{R^{i}(x)}))\subseteq \mathcal{W}^{s}(f_{i}(x),\varepsilon)$, then    $f_{i}(\mathcal{W}^{s}(x,\overline{R^{i}(x)}))\subseteq \mathcal{W}^{s}(f_{i}(x),\overline{R^{i+1}(f_{i}(x))})$.

\medskip

If $\text{int}( R^{i}_{k}) \cap f_{i}^{-1}(\text{int}( R_{j}^{i+1}))\neq \emptyset$, then there exists some $x\in \text{int}( R^{i}_{k}) \cap f_{i}^{-1}(\text{int}( R_{j}^{i+1}))$ such that $R^{i}_{k}=\overline{R^{i}(x)}$ and $R^{i+1}_{j}=\overline{R^{i+1}(f_{i}(x))}.$ If $z\in R_{k}^{i}\cap f^{-1}_{i}(R^{i+1}_{j})$, then   $$\mathcal{W}^{s}(z, R^{i}_{k})=\{[z,y]: y\in \mathcal{W}^{s}(x,R^{i}_{k})\}$$ and
 \begin{align*} f_{i}(\mathcal{W}^{s}(z, R_{k}^{i}))&=\{[f_{i}(z), f_{i}(y)]: y\in \mathcal{W}^{s}(x,R_{k}^{i})\}\subseteq \{[f_{i}(z), w]: w\in \mathcal{W}^{s}(f_{i}(x),R_{j})\}\\
 &\subseteq \mathcal{W}^{s}(f_{i}(z),R_{j}^{i+1}).
 \end{align*}

Analogously we can prove that $  \mathcal{W}^{u}(f_{i}(x), R_{j}^{i+1})\subseteq f_{i}(\mathcal{W}^{u}(x, R_{k}^{i}))$, which completes the proof.
\end{proof}

\begin{proof}[Proof of Theorem D] 
Follows from Theorem \ref{Final}.
\end{proof}

\section{Further Generalizations}

There are several directions to pursue the studies of  Anosov families and questions that still need to be answered. We leave here some topics of interest, and issues that merit attention in the study of this class of dynamical systems.

\begin{enumerate}[(i)]

\item  Verifying if in the case of non-stationary dynamic systems, we can use the shadowing property to have structural stability.

\item Extending the works done in \cite{Fisher05}, \cite{Fisher07} and \cite{Fisher} to the orientation-preserving case, to higher genus surfaces, to higher dimensional tori, and to nonlinear Anosov maps. In \cite{Fisher}, Section 1.6, the authors address these issues in detail.


\item Generalizing Anosov families  to continuous time. In this case we would have a \emph{flow families} instead of  a diffeomorphism families. According to comments and suggestions from \cite{Fisher} Section 1.6, examples of flow families to consider are: (i) the suspension flow of a mapping family and (ii) those given  by nonautonomous differential equations, where the orbits are integral curves of time-varying vector fields. The authors note that an interesting fact in the  suspension of a multiplicative family is that it models the \emph{scenery flow} of the transverse irrational circle rotation. See also \cite{Fisher05} and \cite{Fisher07} for details. \end{enumerate}



\end{document}